\documentclass[reqno]{amsart}
\usepackage[utf8]{inputenc} 
\usepackage{graphicx}
\usepackage[all]{xy}
\usepackage{xcolor}

\usepackage{amsmath}
\usepackage{amssymb}
\usepackage{amsfonts}
\usepackage{amsthm}
\usepackage{multirow}
\usepackage{enumerate}
\usepackage{color}
\usepackage{xcolor}

\usepackage{dsfont}
\usepackage{braket}

\usepackage{setspace}
\newcommand{\N}{\mathbb{N}}
\newcommand{\Z}{\mathbb{Z}}
\newcommand{\R}{\mathbb{R}}
\newcommand{\C}{\mathbb{C}}

\newcommand{\CP}{\mathbb{C}P}

\newcommand{\K}{K\"{a}hler\ }

\renewcommand{\epsilon}{\varepsilon}

\newcommand{\Vol}{\operatorname{Vol}}

\newcommand{\Proj}{\mathbb P}

\newtheorem{theor}{Theorem}[section]

\newtheorem{lem}[theor]{Lemma}
\newtheorem{cor}[theor]{Corollary}

\newtheorem{ex}[theor]{Example}
\newtheorem{rmk}[theor]{Remark}

\begin{document}

\title[Some characterizations of the complex projective space]{Some characterizations of the complex projective space  via Ehrhart polynomials} 

\author{Andrea Loi}
\address{(Andrea Loi) Dipartimento di Matematica \\
         Universit\`a di Cagliari, Via Ospedale 72, 09124  (Italy)}
         \email{loi@unica.it}

\author{Fabio Zuddas}
\address{(Fabio Zuddas) Dipartimento di Matematica \\
         Universit\`a di Cagliari, Via Ospedale 72, 09124  (Italy)}
         \email{ fabio.zuddas@unica.it}

\thanks{
The authors were supported  by STAGE - Funded by Fondazione di Sardegna and by INdAM GNSAGA - Gruppo Nazionale per le Strutture Algebriche, Geometriche e le loro Applicazioni.
}


\begin{abstract}
Let $P_{\lambda\Sigma_n}$  be the Ehrhart polynomial associated to an intergal multiple $\lambda$  of the standard symplex $\Sigma_n\subset \R^n$.
In this paper we  prove that if $(M, L)$ is an $n$-dimensional   polarized toric manifold with associated Delzant polytope $\Delta$ and Ehrhart polynomial $P_\Delta$ such that $P_{\Delta}=P_{\lambda\Sigma_n}$, for some 
$\lambda\in\Z^+$, then $(M, L)\cong (\C P^n, O(\lambda))$ (where $O(1)$ is the hyperplane bundle on  $\CP^n$) in the following three cases: 1. arbitrary $n$ and $\lambda=1$, 2. $n=2$ and $\lambda =3$, 3. $\lambda =n+1$ under the assumption that the polarization $L$ is asymptotically Chow semistable.

\end{abstract}
 
\maketitle

\tableofcontents  

\section{Introduction}
A beautiful fundamental result \cite{Fu} in the theory of toric manifolds states that a pair $(M,L)$ given by a (smooth) compact $n$-dimensional toric manifold $M$ and a very  ample line bundle $L$ on $M$ (i.e. $M$ a polarized toric manifold with polarization $L$) is described combinatorially by a convex polytope $\Delta \subseteq \R^n$, called the {\em Delzant polytope},  having vertices in $\Z^n$ (from now on, we will call such a polytope a {\em lattice polytope}). 
More precisely, in order to represent a smooth polarized toric manifold a lattice  polytope must satisfy the so-called {\it Delzant condition} (cfr. \cite{guillemin}):

\begin{enumerate}
\item[(i)] there are $n$ edges meeting at each vertex $p$;
\item[(ii)] each edge is of the form $p + t v_i$, $i=1, \dots, n$,  where $v_i \in \Z^n$;
\item[(iii)] $v_1, \dots, v_n$  can be chosen to be a basis of $\Z^n$.
\end{enumerate}

Notice that if  $\Delta, \tilde\Delta \subseteq \R^n$ are two {\em unimodularly equivalent} Delzant polytopes, i.e. there exists a matrix $A \in SL_n(\Z)$ and an integral vector $c \in \Z^n$ such that $\tilde \Delta = A(\Delta) + c$, then they represent, up to isomorphism, the same polarized toric manifold.

In order to find some numerical invariants of  a polarized toric manifold  $(M, L)$ one considers 
the polarized toric manifolds $(M, L^m)$, where $L^m$
denotes the $m$-th tensor power of $L$ ($m$ is a natural number). In this case  $(M, L^m)$ corresponds to the dilated polytope $m \Delta$
and the dimension of the space $H^0(M,L^m)$ of global  holomorphic sections of $L^m$   equals the number $\sharp(m\Delta \cap \Z^n)$ of the lattice points belonging to $m \Delta$. By the celebrated Ehrhart theory of convex polytopes, it is known that this number is given by a polynomial

\begin{equation}\label{polEr}
P_{\Delta}(m) = A_n m^n + A_{n-1} m^{n-1} + \cdots + A_1 m + 1
\end{equation}

of degree $n$ in $m$ called the  {\it Ehrhart polynomial of $\Delta$}. 
We refer the reader to \cite{libro} for the meaning of the coefficients $A_j$ (which are independent of $m$ and depend only on the non-dilated $\Delta$). Here we just recall that the leading coefficient $A_n$ equals the euclidean volume $\Vol (\Delta)$ 
of $\Delta$ which in turn is related to the Riemannian volume $\Vol (M)$ of $M$ (with respect to any \K metric $\omega \in c_1(L)$) by 
\begin{equation}\label{volumes}
A_n=\Vol(\Delta) = (2 \pi)^{-n} \Vol(M).
\end{equation}

The prototype and fundamental  example of polarized toric manifold  is the complex projective space equipped with its hyperplane bundle described in the following example.

\begin{ex}\rm\label{expro}
Let us denote by $\Sigma_n$  the $n$-dimensional standard simplex given by the convex hull of $\bar 0, e_1, e_2, \dots, e_n$ (being $e_1, e_2, \dots, e_n$ the canonical basis of $\R^n$). The corresponding polarized toric manifold is $\C P^n$ endowed with the hyperplane bundle $O(1)$.
The Ehrhart polynomial of $\Sigma_n$ is
$$P_{\Sigma_n}(m)= \frac{(m+1)(m+2) \cdots (m+n)}{n!}$$
In particular, in this case $A_n = \frac{1}{n!}$.
Given  $\lambda \in \Z^+$, the convex polytope $\lambda \Sigma_n$ 
representing $(\C P^n, O(\lambda))$ endowed with the $\lambda$-th power of  $O(1)$, has then associated  polynomial 
\begin{equation}\label{projlambda}
P_{\lambda \Sigma_{n}}(m)= \frac{(\lambda m+1)(\lambda m+2) \cdots (\lambda m+n)}{n!}.
\end{equation}
\end{ex}

It is natural and  interesting  to analyze to what extent the Ehrhart polynomial $P_{\Delta}(m)$ determines $\Delta$ and hence the corresponding  polarized toric manifold (see for example Open Problems in Chapter 3 of \cite{libro}, and also \cite{Erbe}, \cite{Haa}, \cite{HibiTsu}).
For this reason one says that two Delzant polytopes $\Delta$ and $\tilde\Delta$ are {\it Ehrhart equivalent} and we write $\Delta\sim\tilde\Delta$, if  their Ehrhart polynomials are equal, i.e. $P_{\tilde \Delta}(m) =P_{\Delta}(m)$.
Of course two  unimodularly equivalent Delzant polytopes $\Delta, \tilde\Delta \subseteq \R^n$ are  Ehrhart equivalent but not viceversa.
However, one can prove (\cite{Haa}, Proposition 4.3) that in that case $\Delta$ and $\tilde \Delta$ are {\it equidecomposable}: this means that there exists two decompositions of $\Delta$ and $\tilde \Delta$ as finite union of lattice polytopes $\Delta = D_1 \cup \cdots \cup D_k$,  $\tilde \Delta = \tilde D_1 \cup \cdots \cup \tilde D_k$ (intersecting at most along their boundaries) such that $D_j$ and $\tilde D_j$ are unimodularly equivalent (with associated matrix $A_j$ and vector $c_j$ depending on $j$).

In   this paper we address the following natural question.

\vskip 0.3cm

\noindent
{\bf Question 1.} {\it Let $\Delta \subseteq \R^n$ be the Delzant polytope representing a smooth  polarized toric $n$-dimensional manifold $(M, L)$. 
Assume that  $\Delta\sim\lambda\Sigma_n$, for some $\lambda\in \Z^+$, where $\Sigma_n$ is given in Example \ref{expro}.
Under which conditions  on the polarization $(M, L)$ can one deduce $\Delta=\lambda\Sigma_n$ or equivalently  $(M, L) \simeq (\C P^n, O(\lambda))$?} 
\vskip 0.3cm

The following two examples show the necessity of imposing some extra conditions in order to achieve the conclusion of Question 1.

\begin{ex}\rm\label{primaprop}
For any  natural number $n$, let $\lambda$ be the smallest common divisor of $2, 3, \dots, n$ 
and consider the $n$-dimensional polarized manifold $(M, L)$  given by the product of $n$-copies of the one-dimensional complex projective space $\CP^1$ where the $j$-th factor has polarization $O(\frac{\lambda}{j})$
with $j=1, \dots , n$, namely
$$(M,L) = \left(\C P^1 \times \cdots \times \C P^1, O\left( \lambda \right) \times O\left( \frac{\lambda}{2} \right) \times \cdots \times O \left( \frac{\lambda}{n} \right)\right).$$

Notice that  the Ehrhart polynomial of the one-dimensional symplex $\left[ 0, \frac{\lambda}{j} \right]$, which represents $\left( \C P^1, O\left( \frac{\lambda}{j} \right) \right)$,  is given by $\frac{\lambda}{j} m + 1$, $j = 1, \dots, n$. Thus, if $\Delta$ is the polytope rappresenting $(M, L)$ and one uses  the general fact that if $\Delta_1$ (resp. $\Delta_2$) is the polytope representing the polarized toric $(M_1, L_1)$ (resp. $(M_2, L_2)$) then $\Delta_1 \times \Delta_2$ represents 
$(M_1 \times M_2, L_1 \times L_2)$ 
one gets
$$P_\Delta (m)= \left( \lambda m + 1 \right) \left( \frac{\lambda}{2} m + 1 \right) \cdots \left( \frac{\lambda}{n} m + 1 \right)= \frac{(\lambda m +1) (\lambda m + 2) \cdots (\lambda m + n)}{n!}$$
and hence, by \eqref{projlambda}, $\Delta\sim\lambda\Sigma_{n}$.
Notice that by taking $n=2$ and $\lambda=2$ 
we have that the polytope $2 \Sigma_2$ representing $(\CP^2, O(2))$ is Ehrhart  equivalent to the polytope representing $(\C P^1 \times \C P^1, O(2) \times O(1))$, i.e. the rectangle with vertices $(0,0), (2,0), (0,1), (2,1)$. In fact, it is easy to show that the polytopes are equidecomposable as defined above.
\end{ex}

\begin{ex}\rm\label{primapropbis}
For another polytope  Ehrhart equivalent to $2\Sigma_2$, take the convex hull $\Delta$ of the vertices $(0,0), (0,1), (1,1), (3,0)$. In fact, it is easy  to see that 
the $P_\Delta =2m^2+3m+1=P_{2\Sigma_2}$, where the last equality follows by \eqref{projlambda}.
This polytope is well-known in the theory of toric surfaces as it represents a very ample line bundle on the so-called {\it Hirzebruch surface} $F_2$, defined as the projectivized line bundle $\Proj(\C + O(-2))$ on $\C P^1$ (the general Hirzebruch surface $F_n$\footnote{The importance of Hirzebruch surfaces in toric geometry lies in the result that all compact toric surfaces are obtained from $\C P^2$ or a Hirzebruch surface by blowing ups (in points fixed by the torus action). Let us notice that Hirzebruch surfaces $F_n$, $n \geq 2$, are not Fano. Indeed the only toric Fano surfaces are $\C P^2, \C P^1 \times \C P^1$ and the blowups of $\C P^2$ at one, two or three points fixed by the torus action; $F_1$ coincides with the blow up of $\C P^2$ at one point.} is just the projectivized line bundle $\Proj(\C + O(-n))$ on $\C P^1$, represented by the polytope having vertices $(0,0), (0,1), (1,1), (n+1,0)$).
\end{ex}

The main result of this paper is the following theorem where we prove that Question 1 has an affirmative answer in three interesting cases: arbitrary $n$ and $\lambda=1$ (hyperplane bundle), $n=2$ and $\lambda=3$
(namely  a polarization of a toric surface  with the same Ehrhart polynomial  of 
anticanonical bundle over $\CP^2$) and  finally arbitrary $n$,  $\lambda =n+1$ (anticanonical bundle) under the additional assumption that the polarization $(M, L)$ is asymptotically Chow semistable (in the sequel ACsemiS) (see Section \ref{proofs}  for details).
\begin{theor}\label{mainteor}
Let $\Delta \subseteq \R^n$ be the Delzant polytope representing a polarized toric $n$-dimensional manifold $(M, L)$. 
The then following facts hold true.
\begin{itemize}
\item [(i)]
if $\Delta\sim\Sigma_n$ then $(M,L) \simeq (\C P^n, O(1))$;
\item [(ii)]
if $\Delta\sim 3\Sigma_2$ then $(M,L) \simeq (\C P^2, O(3))$;
\item[(iii)]
if $\Delta\sim (n+1)\Sigma_n$ and $(M,L)$ is a ACsemiS then  
$(M,L)\simeq (\C P^n, O(n+1))$.
\end{itemize}
\end{theor}

Notice that  the assumption of  ACsemiS  of $(M, L)$  in (iii) of Theorem \ref{mainteor} 
cannot be dropped (see Example \ref{exnc} in Section \ref{secEhrart} below).

The paper is organized as follows. In Section \ref{secEhrart} we  recall some  basic facts on reflexive polytopes and barycenters needed in the proof of Theorem \ref{mainteor}.
In Section \ref{proofs} after recalling various definitions of stability for a given polarization,  their links  with the regular  quantization of \K manifolds and the theory of balanced metrics in Donaldsons' terminology, we prove Theorem \ref{mainteor}.

\section{Reflexive polytopes and barycenters}\label{secEhrart}
An important notion in the theory of lattice poytopes is that of {\it dual polytope}.
Let us recall that given a convex polytope $\Delta \subseteq \R^n$ such that the origin $\bar 0$ is contained in its interior, the dual polytope $\Delta^*$ is defined as

$$\Delta^* = \{ x \in \R^n \ | \ \langle x, y \rangle \leq 1 \ \textrm{for any} \  y \in \Delta \}$$

Equivalently, one has 

$$\Delta^* = \{ x \in \R^n \ | \ \langle x, v \rangle \leq 1 \ \textrm{for any vertex} \  v \ \textrm{of} \  \Delta \}$$
(the equivalence is easily seen by using the fact that $\Delta$ is the convex hull of its vertices). Notice that  $(\Delta^*)^* = \Delta$.

\begin{ex}\rm\label{exdual}
Take the simplex\footnote{It is just the simplex $\lambda \Sigma_2$ having as vertices $(0,0), (\lambda, 0), (0, \lambda)$ translated by the vector $(-1,-1)$ in order to have the origin in its interior as required by the definition of dual.} $\Sigma$ in $\R^2$ given by the convex hull of its vertices $v_1 = (-1,-1), v_2 = (\lambda -1, -1), v_3 = (-1, \lambda - 1)$.
It is easy to see that the vertices of $\Sigma^*$ are then given by the points $(-1,0)$, $(0, -1)$, 
$\left( \frac{1}{\lambda-2}, \frac{1}{\lambda-2} \right)$. Notice that $\Sigma^*$ is in general a rational polytope and it is a lattice polytope if and only if $\lambda=3$.
\end{ex}

A lattice polytope $\Delta$ for which the dual $\Delta^*$ is again a lattice polytope is called {\it reflexive}. 
It is well-known that reflexive polytopes correspond to toric (Fano) varieties with the anticanonical polarization (see, e.g. \cite{Bat}). For example, the above example shows that $3\Sigma_{2} \subseteq \R^2$ is reflexive and, in general,  the simplex  $\lambda \Sigma_{n} \subseteq \R^n$ described in the introduction is reflexive if and only if $\lambda = n+1$ (indeed  the anticanonical bundle of $\C P^n$ is given by $O(n+1)$).
An interesting   criterium for a lattice polytope $\Delta \subseteq \R^n$ in order to be reflexive in terms of the Ehrhart polynomial is given by the following theorem, due to Hibi \cite{Hibimain} (see also Theorem 3 and Remark 2.1 in \cite{St}).

\begin{theor}\label{theorcriteriumREFL}
A lattice polytope $\Delta \subseteq \R^n$ with $\bar 0 \in \Delta$
is reflexive  if and only if
$$P_{\Delta}(m) = (-1)^{n} P_{\Delta}(-m-1)$$
for every $m \in \N$.
\end{theor}

The Ehrhart polynomial can be considered a particular case in a more general class of functions defined on the dilations of a polytope. Let $\phi$ be a polynomial map homogeneous of degree $k$ on $\R^n$ and consider the function

\begin{equation}\label{genn}
P_{\Delta}^{\phi}(m):= \sum_{x \in (m\Delta) \cap \Z^n} \phi(x)
\end{equation}

The Ehrhart polynomial corresponds to choosing $\phi$ constant equal to $1$. One can show that, analogously to the Ehrhart case, one always gets a polynomial on $m$ of degree $n+k$ and that the following property  holds
\begin{equation}\label{reciprocPHI}
(-1)^{n+k} P^{\phi}_{\Delta}(-m) = P_{\Delta^0}^{\phi}(m):=\sum_{x \in (m\Delta)^0 \cap \Z^n} \phi(x),
\end{equation}  
where $(m\Delta)^{\circ}$ denotes the interior of $m\Delta$
(see \cite[Proposition 4.1]{BrionVergne}).
In particular when $\phi (x)=1$ one gets 
\begin{equation}\label{reciproc}
(-1)^n P_{\Delta}(-m) = \sharp((m\Delta)^{\circ} \cap \Z^n)
\end{equation}
This is  called {\em Ehrhart-MacDonald reciprocity} (for a direct  proof, see e.g.  \cite[Theorem 2]{St}).

\begin{rmk}\rm\label{onepoint}
Notice that  by (\ref{reciproc}) and Theorem \ref{theorcriteriumREFL} with $m=0$ one deduces that a reflexive polytope has exactly one lattice point in its interior.
\end{rmk}

On the other hand, by choosing $\phi(x) = x$, one can deduce  the following:
\begin{equation}\label{phi=x}
P_{\Delta}^{x}(m):= \sum_{x \in (m\Delta) \cap \Z^n} x =\left( \int_{\Delta} x dv \right) m^n +  \left( \frac{1}{2}\int_{\partial \Delta} x da \right) m^{n-1} + \cdots
\end{equation}

Let us recall that the barycenter $b_{\Delta}$ of a polytope $\Delta \subseteq \R^n$ is the point of $\R^n$ given by:
$$b_{\Delta} =\frac{\int_{\Delta} x  dv}{\Vol(\Delta)}$$

so the leading coefficient of $P_{\Delta}^{x}(m)$ is $\Vol(\Delta) \cdot b_{\Delta}$.
The  coefficient of $m^{n-1}$ is a multiple  of the barycenter of the boundary $\partial \Delta$ of $\Delta$.
It is interesting to notice (even if it is not used in the rest of the paper) that these two barycenters are in general different: they coincide if and only if the  so-called  {\it Futaki invariant} of $(M,L)$ vanishes (see \cite{LeBr}).

We conclude this section with the following result needed in the proof of (iii) of Theorem \ref{mainteor}.

\begin{lem}\cite[Theorem 1.4]{NP}\label{lemfond}
Let $\Delta \subseteq \R^n$ be a convex polytope such that:
\begin{enumerate}
\item[(a)] $\Delta$ is contained in the dual $D^*$ of a lattice polytope $D$;
\item[(b)] the barycenter $b_\Delta$ of $\Delta$ is the origin, i.e.  $b_\Delta=\bar 0$.
\end{enumerate}

Then $\Vol(\Delta) \leq \frac{(n+1)^n}{n!}$ and the equality is attained if and only if $\Delta = (n+1)\Sigma_{n}$.
\end{lem}

\section{The proof of Theorem \ref{mainteor}}\label{proofs}
Let $(M, L)$ be a polarization of a compact complex manifold $M$.
In order to analyze the existence of a constant scalar curvature \K\  (cscK) metric  $\omega\in c_1(L)$ many concepts  of  {\em  stability} of $(M, L)$ have been introduced.
Here we recall the following (starting from the weaker to 
 the stronger assumption) 1. {\it asymptotically Chow semistability (ACSemiS)}, 2. {\it asymptotically Chow polystability (ACpolS)}. and  3. {\it asymptotically Chow stability (ACS)}. The definitions of these notions are omitted since they  require some technical tools such as that of test configuration and Chow weight not treated in the present paper (the interested reader is referred  to \cite{DVZ} for details).
 From the point of view of Geometric Invariant Theory,  the ACpolS of $(M, L)$ is equivalent to the existence of a balanced metric $\omega_m \in c_1(L^m)$ for any $m >>0$. 
Following Donaldson \cite{do} (see also \cite{arlcomm}) one says  that a \K metric $\omega_m\in c_1(L^m)$ on a compact polarized manifold $(M,L)$ is said to be {\em balanced} if the so called {\it Kempf distortion function} 
$$T_{\omega_m}(p) = \sum_{j=0}^{N_m} h_m(s_j(p), s_j(p))$$ is a positive  constant, 
where $h_m$ is a hermitian metric on $L^m$ such that $Ric(h_m) =\omega_m$ and $s_0, \dots, s_N$ form a orthonormal basis of $H^0(M,L^m)$ with respect to the $L^2$-scalar product $\langle s, t \rangle = \int_M h_m(s(p), t(p)) \frac{\omega^n}{n!}$.
Another notion stronger than ACpolS is the following.  A polarization $(M, L)$
of a compact complex manifold $M$ is said to be a  {\em regular quantization} (originally introduced in \cite{cgr2}, see also \cite{arlhom})  if there exists
$\omega	\in c_1(L)$ such that $m\omega$ is balanced for  $m >>0$.
Of course a regular quantization is ACpolS and hence ACsemiS. 
Hence Theorem \ref{mainteor} immediately  implies the following:

\begin{cor}\label{cormain}
Let $\Delta\sim(n+1)\Sigma_{n}$ and  assume that $(M, L)$ is a regular quantization. 
Then $(M,L) \simeq (\C P^n, O(n+1))$.
\end{cor}

\begin{rmk} \rm
Actually  an alternative and more direct  proof of Corollary \ref{cormain} can be obtained as follows. 
By $\Delta\sim (n+1)\Sigma_n$ we deduce as in the proof below of (iii) in Theorem \ref{mainteor} that  the polytope $\Delta$ is reflexive and hence $L$ is the anticanonical bundle.
Let $\omega\in c_1(L)$ be a \K\ metric such that $m\omega$ is balanced for  $m >>0$. 
Then $\omega\in c_1(L)$ is a cscK metric (see \cite{Lreg} for a proof) and hence  Kahler-Einstein (being  a  cscK metric in the first Chern class of the anticanonical bundle).
Then, the conclusion follows immediately by Proposition 1.3 of \cite{NP}.
\end{rmk}

\begin{rmk} \rm
Notice that the assumption $\Delta\sim(n+1)\Sigma_n$ in (iii) of Theorem \ref{mainteor} cannot be dropped even if one assumes the stronger assumptions of  ACS or regular quantization.   For example,
the blow-up of $\C P^2$ at three points endowed with the anticanonical polarization is an
 ACS compact smooth toric manifold (see Corollary 3.3 in \cite{LeeLietc}).
 For another example let us  consider Example \ref{primaprop} above, where $\Delta\sim \lambda\Sigma_n$ with $\lambda\neq n+1$.
 It is not hard to see that in this case the polarization $(M, L)$ is indeed regular being a flag manifold, i.e. a  homogeneous and simply-connected \K\ manifold (see \cite{arlcomm} for a proof).
 We believe that if $M$ is a compact toric manifold admitting  a regular quantization  
then $M$ is isomorphic to the \K\ product of complex projective spaces.
\end{rmk}

The following example shows the necessity of ACsemiS in Theorem \ref{mainteor}.
\begin{ex}\rm\label{exnc}
Consider the lattice polytope $\Delta$ in $\R^5$ with vertices
$$(0,-1,-1,-1,-1), (-1,0,-1,-1,-1), (-1,-1,-1,-1,-1,), (-1,0,0,-1,-1),$$
$$(-1,-1,0,-1,-1), (0,-1,0,-1,-1), (0,-1,0,2,-1), (-1,-1,-1,2,-1), $$
$$(6,-1,-1,-1,2), (-1,6,3,-1,2), (6,-1,3,-1,2),(-1,-1,0,2,-1),$$ 
$$(-1,-1,3,-1,2), (-1,0,0,2,-1), (-1,0,-1,2,-1),(-1,6,-1,-1,2),$$
$$(0,-1,-1,2,-1), (-1,-1,-1,-1,2).$$
This polytope represents a compact smooth toric manifold $M$  (different from $\C P^5$)  polarized by the anticanonical bundle $L=-K_M$  and $\Delta\sim 6\Sigma_5$ (this can be checked by calculating the Ehrhart polynomials of the corresponding polytopes by the software Normaliz, https://www.normaliz.uni-osnabrueck.de/; see also the database on smooth toric Fano varieties http://www.grdb.co.uk/forms/toricsmooth, setting dimension 5 and degree $(n+1)^n = 6^5$). We thank Akiyoshi Tsuchiya for pointing out to us this example. 
Notice that $(M, L)$ is not ACsemiS by Theorem \ref{mainteor}.
\end{ex}

A necessary condition needed in the proof of (iii) of Theorem \ref{mainteor}
for the   ACsemiS of a polarized toric manifold $(M,L)$  can be expressed as follows in terms of the Delzant polynomial $\Delta$ associated to $(M, L)$
and its barycenter.

\begin{lem}(\cite[Theorem 1.2]{Ono})\label{ACS}
Assume that $(M, L)$ is  ACsemiS. Then for any integer $m \in \Z$
one has
\begin{equation}\label{yots} 
P_{\Delta}^{x}(m) = P_{\Delta}(m)b_\Delta,
\end{equation} 
where $b_\Delta$ is the barycenter of $\Delta$ and  we take $\phi (x)=x$ in \eqref{genn}.
\end{lem}

Now we are ready to prove Theorem \ref{mainteor}.

\begin{proof}[Proof of Theorem \ref{mainteor}]
\vskip 0.1cm
{\em Proof of (i)}
The assumption  $P_{\Delta}(m)=P_{\Sigma_n}(m)$ implies in particular that $\Vol (\Delta)=\Vol (\Sigma_n)=\frac{1}{n!}$ .
But, since $\Delta$ is a lattice polytope, it is easy to see that this implies that $\Delta$ must coincide with the simplex $\Sigma_n$.
Indeed, since $(M, L)$ is a toric manifold, $\Delta$ is a Delzant polytope and, up to an $SL_n(\Z)$ transformation and a translation by a lattice vector we can assume that $\bar 0$ is a vertex of $\Delta$ and that the edges at $\bar 0$ are given by the directions of $e_1, \dots, e_n$, i.e. $\Delta$ must contain $c_1 e_1, \dots, c_n e_n$ as vertices, for some $c_1, \dots, c_n \in \Z^+$. Since $\Delta$ is a convex lattice polytope, it must contain the simplex $\Sigma$ given by the convex hull of $\bar 0, c_1 e_1, \dots, c_n e_n$, and then

$$\frac{1}{n!} = \Vol(\Delta) \geq \Vol(\Sigma) = \frac{c_1 \cdots c_n }{n!}$$

Then $c_1 = \cdots = c_n = 1$, i.e. $\Delta$ must contain $e_1, \dots, e_n$ as vertices and, since its volume coincide with that of $\Sigma_n$, it cannot have other points, i.e. $\Delta = \Sigma_n$ and then (i) 
is proved.

\vskip 0.1cm
{\em Proof of (ii)}
By setting $n=2$ and $\lambda=3$ in \eqref{projlambda} (Example \ref{expro} in the introduction)  one gets
$$
P_\Delta (m)=P_{3\Sigma_2}(m)=\frac{(3 m+1)(3 m+2)}{2}
$$
In particular 
$$\Vol (\Delta)=\frac{9}{2}$$
and 
$$P_{\Delta}(-1)=P_{3\Sigma_{2}}(-1)=(-1)^2=1.$$
Combining the latter with Ehrhart-MacDonald reciprocity \eqref{reciproc} one deduces that  
$\Delta$ has exactly one lattice point in its interior.
Then the desired conclusion, namely the isomorphism $(M, L)\cong (\C P^2, O(3))$, follows by 
\cite[Corollary 4.2]{PhDBall} which states that {\em for any two-dimensional polytope $\Delta$ with exactly one point in its interior the volume $\Vol(\Delta)$ must be less or equal to $\frac{9}{2}$ and the equality holds if and only if $\Delta = 3 \Sigma_{2}$}.

\vskip 0.1cm
{\em Proof of (iii)}
We show  that $(a)$ and $(b)$ in Lemma \ref{lemfond} hold for the polytope $\Delta$ representing $(M, L)$. By the assumption
$\Delta\sim(n+1)\Sigma_{n}$  and the reflexivity of  $(n+1)\Sigma_n$  we deduce that 
$$P_{\Delta}(m) = P_{(n+1)\Sigma_{n}}(m) = (-1)^{n} P_{(n+1)\Sigma_{n}}(-m-1) =(-1)^n P_{\Delta}(-m-1),$$
which, by Theorem \ref{theorcriteriumREFL}, yields the reflexivity of $\Delta$.
Thus, by the very definition of reflexivity  $\Delta^*$ is a lattice polytope and hence  $(a)$ of Lemma \ref{lemfond} is satisfied with $D = \Delta^*$ (since $\Delta = (\Delta^*)^*$).

Notice now that by Remark \ref{onepoint} the polytope  $\Delta$ has only one lattice point in the interior and so, up to an integral translation (which does not change the Ehrhart polynomial), 
we can assume that this point is the origin $\bar 0$. Then, by (\ref{reciprocPHI}) 
$$(-1)^{n+1} P^x_\Delta (-1)=P^x_{\Delta^0}(1)=\bar 0.$$
Using  the assumption of ACsemiS and (\ref{yots}) in Lemma \ref{ACS} with $m=-1$
we get
$$\bar 0 = P_{\Delta}^{x}(-1) = P_{\Delta}(-1)b_\Delta.$$
On the other hand,  by (\ref{reciproc}) with $m=1$, we have 
$$P_{\Delta}(-1) = (-1)^n \sharp(\Delta^0 \cap \Z^n) = (-1)^n \neq 0,$$
 and then one concludes $b_\Delta= \bar 0$. This shows that also $(b)$ of Lemma \ref{lemfond} is satisfied.
Thus, since $\Delta\sim (n+1)\Sigma_n$ implies
$\Vol (\Delta)=\Vol ((n+1)\Sigma_n)=\frac{(n+1)^n}{n!}$ (cfr. \eqref{volumes}) 
Lemma \ref{lemfond}  yields  the desired equality $\Delta = (n+1)\Sigma_{n}$, i.e. $(M, L)\cong (\C P^n, O(n+1))$. 
The proof of (iii) and hence of the theorem is concluded.
\end{proof}


\begin{thebibliography}{99}



\bibitem{arlcomm} C. Arezzo and A. Loi,
{\em Moment maps, scalar curvature and
quantization of K\"{a}hler manifolds},
Comm. Math. Phys. 246 (2004), 543-549.


\bibitem{arlhom} C. Arezzo and A. Loi,
{\em On homothetic balanced metrics},
Ann.  Global Anal.  Geom. 41, n. 4  (2012), 473-491.


\bibitem{PhDBall} G. Balletti, {\em  Classifications and volume bounds of lattice polytopes}, Ph. D. Thesis (Stockholm university), 2017

\bibitem{Bat} V. Batyrev, {\em  Dual polyhedra and mirror symmetry for Calabi-Yau hypersurfaces in toric varieties}, J. Algebraic Geom., Vol. 3 (1994), 493-535.



\bibitem{libro} M. Beck, S. Robins, {\em  Computing the Continuous Discretely}, Springer 2009

\bibitem{BrionVergne} M. Brion, M. Vergne, {\em  Lattice points in simple polytopes}, Journal of the
American Mathematical Society, Vol. 10, No. 2 (1997), 371-392.

\bibitem{cgr2} M. Cahen, S. Gutt, J. H. Rawnsley,
{\em Quantization of K\"{a}hler manifolds II}, Trans. Amer. Math.
Soc. 337 (1993), 73-98.


\bibitem{DVZ} A. Della Vedova, F. Zuddas, {\em Scalar curvature and asymptotic Chow stability of projective bundles and blowups},  Trans. Amer. Math. Soc. 364 (2012) no. 12, 6495-6511. 

\bibitem{do} S. Donaldson,
{\em Scalar Curvature and Projective Embeddings, I},
J. Diff. Geometry 59 (2001),  479-522.




\bibitem{Erbe}  J. Erbe, C. Haase and F. Santos {\em Ehrhart-equivalent 3-polytopes are equidecomposable}, Proc. Amer. Math. Soc. 147 (2019), 5373-5383

\bibitem{Fu} W. Fulton, {\em Introduction to toric varieties}, Princeton University Press 1993.

\bibitem{guillemin}
V. Guillemin,
Moment maps and combinatorial invariants of Hamiltonian $T^n$-spaces,
Progress in Mathematics, 122. Birkhauser Boston, Inc., Boston, MA, 1994. 

\bibitem{Haa} C. Haase and T. B. McAllister, {\em Quasi-period collapse and GLn(Z)-
scissors congruence in rational polytopes}, Integer Points in Polyhedra-Geometry, Number Theory, Representation Theory, algebra, Optimization, Statistics, Contemporary Mathematics vol. 452, Amer. Math, Soc., Providence, RI (2008), 115–122.


\bibitem{Hibimain} T. Hibi,  {\em Note dual polytopes of rational convex polytopes}, Combinatorica, Vol. 12 (1992), 237-240. 

\bibitem{HibiTsu} T. Hibi, A. Tsuchiya {\em Facets and volume of Gorenstein Fano polytopes}, Mathematische Nachrichten, Vol. 290, Issue 16 November 2017, 2619-2628.

\bibitem{LeBr} C. LeBrun, {\em Calabi Energies of Extremal Toric Surfaces}, in Surveys in Differential Geometry, vol. XVIII:
Geometry and Topology, Cao and Yau, editors, International Press of Boston (2013), 195-226.


\bibitem{LeeLietc} K.-L. Lee, Z. Li, J. Sturm, X. Wang, {\em  Asymptotic Chow stability of toric
Del Pezzo surfaces}, Math. Res. Lett., Vol. 26, No. 6 (2019), 1759-1787.

\bibitem{Lreg} 
A. Loi, 
{\em Regular quantizations of \K\ manifolds and constant scalar curvature metrics},
J. Geom. Phys. 53 (2005), 354-364.





\bibitem{NP} B. Nill, A. Paffenholz {\em On the equality case in Ehrhart's volume conjecture}, Advances in Geometry, 14 (2014), 579-586. 

\bibitem{Ono} H. Ono, {\em A necessary condition for Chow semistability of polarized toric manifolds}, J. Math. Soc. Japan. 63 (2011), 1377-1389.





\bibitem{St} C. Steinert {\em Reflexivity of Newton-Okunkov bodies of partial flag varieties}, arxiv 1902.07105v2 (2020)













\end{thebibliography}
\end{document}